\documentclass[12pt,reqno]{article}
\usepackage{amssymb,latexsym,amsmath,amsthm,mathtools} 
\usepackage{url}
\usepackage[usenames]{color}
\usepackage{letltxmacro}
\usepackage{fullpage}
\makeatletter

\renewcommand\section{\@startsection {section}{1}{\z@}
{-30pt \@plus -1ex \@minus -.2ex}
{2.3ex \@plus.2ex}
{\normalfont\normalsize\bfseries\boldmath}}

\renewcommand\subsection{\@startsection{subsection}{2}{\z@}
{-3.25ex\@plus -1ex \@minus -.2ex}
{1.5ex \@plus .2ex}
{\normalfont\normalsize\bfseries\boldmath}}

\renewcommand{\@seccntformat}[1]{\csname the#1\endcsname. }

\makeatother

\newtheorem{theorem}{Theorem}
\newtheorem{lemma}{Lemma}

\newtheorem{corollary}{Corollary}

\theoremstyle{definition}
\newtheorem{definition}{Definition}

\newtheorem{remark}{Remark}

\usepackage{hyperref}
\usepackage{orcidlink}

\definecolor{webgreen}{rgb}{0,.5,0}
\definecolor{webbrown}{rgb}{.6,0,0}

\newcommand{\seqnum}[1]{\href{https://oeis.org/#1}{\rm \underline{#1}}}

\makeatletter
\let\oldr@@t\r@@t
\def\r@@t#1#2{%
\setbox0=\hbox{$\oldr@@t#1{#2\,}$}\dimen0=\ht0
\advance\dimen0-0.2\ht0
\setbox2=\hbox{\vrule height\ht0 depth -\dimen0}%
{\box0\lower0.4pt\box2}}
\LetLtxMacro{\oldsqrt}{\sqrt}
\renewcommand*{\sqrt}[2][\ ]{\oldsqrt[#1]{#2}}

\begin{document}

\begin{center}
\vskip 1cm{\LARGE\bf 
On the relation between perfect powers and tetration frozen digits}
\vspace{12mm}

{\Large Marco Rip\`a\,\orcidlink{0000-0002-6036-5541}}
\vspace{3mm}

Rome, Italy \\
e-mail: \url{marcokrt1984@yahoo.it}
\end{center}
\vspace{12mm}

\begin{abstract}
\noindent This paper provides a link between integer exponentiation and integer tetration since it is devoted to introducing some peculiar sets of perfect powers characterized by any given value of their constant congruence speed, revealing a fascinating relation \linebreak between the degree of every perfect power belonging to any congruence class modulo $20$ and the number of digits frozen by these special tetration bases, in radix-$10$, for any unit increment of the hyperexponent. In particular, given any positive integer $c$, we constructively prove the existence of infinitely many $c$-th perfect powers that have a constant congruence speed of $c$.
\end{abstract}

\vspace{6mm}


\section{Introduction} \label{sec:Intr}

We explore a tool in modular arithmetic that predicts patterns among the right-hand digits of terms in rapidly growing sequences arising from power towers, thereby reducing computational overhead.

For clarity, we will denote $\mathbb{N}_0$ as the set of nonnegative integers (including zero) and $\mathbb{N}$ as the set of positive integers $\{1,2,3,\ldots\}$.

In recent years, by assuming the standard decimal numeral system (radix-$10$), we have shown that the integer tetration $^{b}a \coloneq \begin{cases} a \mbox{ }\mbox{ }\mbox{ }\mbox{ }\mbox{ }\mbox{ }\mbox{ }\mbox{ }\mbox{ \hspace{1mm} if $b = 1$} \\
a^{(^{(b-1)}a)} \mbox{ \hspace{1mm} if $b\geq 2$}
\end{cases}\hspace{-3.5mm}$ has a unique property \cite{44, 2} involving the number of new frozen rightmost digits for any unit increment of its hyperexponent, $b \in \mathbb{N}$ \cite{11, 12}.
Indeed, this value no longer depends on $b$ as $b$ becomes sufficiently large (see the sequence \seqnum{A372490} in the On-Line Encyclopedia of Integer Sequences \cite{oeis}) and the tetration base, $a \in \mathbb{N}$, is not a multiple of $10$.

We refer to the mentioned property as the constancy of the \textit{congruence speed} of tetration (see Definition~\ref{def2.1} and also the comments of the OEIS sequence \seqnum{A317905}).

From \cite{11}, we know that each positive integer $a \geq 2$, which is not a multiple of $10$, is characterized by a finite, strictly positive, integer value of its constant congruence speed (the map of the constant congruence speed of every $a$ as above is provided by \cite{11, 12}).

Then, the present paper aims to constructively prove the existence of infinitely many perfect powers having any given positive constant congruence speed.

A pleasant result, that follows from Theorem~\ref{Theorem 2.4} as a corollary, is the existence, for any given positive integer $c$, of infinitely many $c$-th perfect powers (i.e., an integer $a > 1$ is a $c$-th perfect power if there exist some integers $\tilde{a}$ and $c$ such that $a = \tilde{a}^c$, so we have \textit{perfect squares} if $c=2$, \textit{perfect cubes} if $c=3$, and so forth) having a constant congruence speed of $c$.


\section{Preliminary investigations with the automorphic numbers} \label{sec:2}

In order to present the results compactly, let us properly define the constant congruence speed of tetration as already done in \cite[p. 442]{12}, Definitions 1.1 and 1.2.

\begin{definition} \label{def2.1}
Let $n\in\mathbb{N}_0$ and assume that $a\in\mathbb{N}-\{1\}$ is not a multiple of $10$. Then, given $^{b-1}a\equiv{^{b}a}\pmod {10^{n}} \, \wedge \, ^{b-1}a \not\equiv{^{b}a}\pmod {10^{n+1}}$, for all $b~\in~\mathbb{N}$, $V(a,b)$ returns the nonnegative integer such that $^{b}a\equiv{^{b+1}a}\pmod {10^{n+V(a,b)}}$ $\wedge$ $^{b}a \not\equiv{^{b+1}a}\pmod {10^{n+V(a,b)+1}}$, and we define $V(a,b)$ as the {\it congruence speed} of the base $a$ at the given height of its hyperexponent $b$.

Furthermore, let $\bar{b} \coloneq \min\left\{b \in \mathbb{N}: V(a, b)=V(a, b+k) \hspace{1mm} \textnormal{for all} \hspace{1mm} k \in \mathbb{N}\right\}$ so that we define, as {\it constant congruence speed} of $a$, the positive integer $V(a) \coloneq V(a, \bar{b})$.
\end{definition}

In general, we know that a sufficient but not necessary condition for having $V(a)=V(a, \bar{b})$ is to set $\bar{b} \coloneq a+1$, and for a tighter bound on $\bar{b} \coloneq \bar{b}(a)$, holding for any $a \not\equiv 0 \pmod {10}$, see \cite[p. 450]{12}.

As a clarifying example, let us consider the case $a=807$. Then, we have $V(807, 1)=0$, $V(807, 2)=4$, $V(807, 3)=4$, $V(807, 4)=4$, $V(807, 5)=4$, and finally $V(807, 6)=V(807, 7)=\cdots=V(807)=3$ since $\tilde{\nu}_{5}\left(807^2+1\right) + 2= 4+2 = 6$ (by \cite[Definition 2.1]{12}) and

\noindent$^{1}807=807$;

\noindent$^{2}807 \equiv 54962039628331827388850173\color{cyan}{\bf{7943}}$ $\pmod {10^{30}}$;

\noindent$^{3}807 \equiv 6016926514668229405256\color{cyan}{\bf{3285}}\color{red}7943$ $\pmod {10^{30}}$;

\noindent$^{4}807 \equiv 146336906474874632\color{cyan}{\bf{6260}}\color{red}32857943$ $\pmod {10^{30}}$;

\noindent$^{5}807 \equiv 35503490744897\color{cyan}{\bf{3150}}\color{red}626032857943$ $\pmod {10^{30}}$;

\noindent$^{6}807 \equiv 47863568981\color{cyan}{\bf{228}}\color{red}3150626032857943$ $\pmod {10^{30}}$;

\noindent$^{7}807 \equiv 02704888\color{cyan}{\bf{876}}\color{red}2283150626032857943$ $\pmod {10^{30}}$;

\hspace{0.1mm}\vdots

\noindent$^{(1 \textnormal{googolplex})}807 \equiv \color{red} 803001638762283150626032857943$ $\pmod {10^{30}}$.

\begin{lemma}\label{Lemma 2.0}
Let $a \in \mathbb{N}$ be such that $a \not\equiv \hspace{-0.5 mm} 0 \pmod {10}$. Then, for all $t \in \mathbb{N}_0$, there exist infinitely many $c \in \mathbb{N}$ such that $V(a^c)=t$.
\end{lemma}

\begin{proof}
Disregarding the special case $t=0$, this proof immediately follows from Definition~1. 

For any integer $a>1$ which is not a multiple of $10$, the constant congruence speed of the tetration $^{b}a$ is well-defined and it is the same for any $b \in \{a+1, a+2, a+3, \dots\}$. Thus, by the last line of Equation~(2) in \cite{11}, it is sufficient to consider $\hat{a} \coloneq 10^t-1$ so that $V(\hat{a})=t$ is true for any given positive integer $t$, and then we can easily complete the proof by observing that $V(1)=0$ as stated in \cite[Definition 1.3]{12}.

Trivially, $V(\hat{a}, b)=V(\hat{a}, b+1)=V(\hat{a}, b+2)=\cdots$ is certainly true for every integer $b \geq \hat{a}+1$ and, by assuming that $t \in \mathbb{N}$, for all the aforementioned values of $b$, we have that $\hat{a}=10^t-1$ implies $V(\hat{a}, b)=t$ (while from $a=1 \Rightarrow t=0$ it follows that $V(1^c)=0$ for any nonnegative integer $c$).

Consequently, let $t \in \mathbb{N}$, assume $b \in \{10^t, 10^t+1, 10^t+2, \dots\}$, and then $V((10^t-1)^c, b)=V((10^t-1)^c)=t$ is true for any $c \in \{{^{b-1}(10^t-1)},{^{b}(10^t-1)},{^{b+1}(10^t-1)},\dots \}$ so that the proof of Lemma~\ref{Lemma 2.0} is complete.
\end{proof}

Thus, Lemma~\ref{Lemma 2.0} shows the existence of infinitely many $c$-th powers of the tetration base $a : a \equiv 1,2,3,4,5,6,7,8,9 \pmod {10}$ that are characterized by any given (arbitrarily large) nonnegative constant congruence speed.

\begin{remark}\label{Remark 0} We note that, in radix-$10$, there exist only three positive $1$-automorphic numbers and they are congruent modulo $100$ to $1$, $25$, and $76$ (respectively). Thus, the corresponding three integers found by considering the two rightmost digits of the analogous solutions of the fundamental $10$-adic equation $y^5=y$, by \cite[Equation~(2)]{11} (see also the OEIS sequences \seqnum{A018247} and \seqnum{A018248}), describe $1$-automorphic numbers \cite{16} (e.g., $\alpha_{76} \mapsto a_{76} \coloneq 76$ since $76^2 \equiv 76 \pmod {10^2}$ and we know that \cite{5}, in radix-$10$, there are only four $10$-adic solutions, including $\alpha_{00} \coloneq \dots000000$, to the equation $y^2=y$). Consequently, by looking at lines 4, 5, and 7 of Equation~(16) in \cite{12}, we can see that the recurrences described by Equations~(\ref{eq1})--(\ref{eq3}) hold for every $c \in \mathbb{N}$.

\begin{equation}\label{eq1}
\tilde{a} \equiv 6 \pmod {10} \Rightarrow V(\tilde{a})\begin{cases}=V(\tilde{a}^c) \quad \textnormal{iff} \quad c \equiv 1,2,3,4 \pmod {5}\\
\leq V(\tilde{a}^c) \quad \textnormal{iff} \quad c \equiv 0 \pmod {5}
\end{cases}.
\end{equation}

\begin{equation}\label{eq2}
\hspace{-1.15cm} \tilde{a} \equiv 5 \pmod {20} \Rightarrow V(\tilde{a})\begin{cases}=V(\tilde{a}^c) \quad \textnormal{iff} \quad c \equiv 1 \pmod {2}\\
\leq V(\tilde{a}^c) \quad \textnormal{iff} \quad c \equiv 0 \pmod {2}
\end{cases}.
\end{equation}

\begin{equation}\label{eq3}
\tilde{a} \equiv 1 \pmod {20} \Rightarrow V(\tilde{a})\begin{cases}=V(\tilde{a}^c) \quad \textnormal{iff} \quad c \equiv 1,2,3,4 \pmod {5}\\
\leq V(\tilde{a}^c) \quad \textnormal{iff} \quad c \equiv 0 \pmod {5}
\end{cases}.
\end{equation}
\end{remark}

The investigation of this observation (with specific reference to Equation~(\ref{eq2})) leads us to the following theorem.

\begin{theorem}\label{Theorem 2.3}
For each $c \in \mathbb{N}$, there exist infinitely many $a : a \equiv 5 \hspace{-0.5 mm} \pmod {20}$ such that $\sqrt[c]{a}\in \mathbb{N} \wedge V(\sqrt[c]{a}) = t \wedge V(a) \geq t$ holds for all $t \in \mathbb{N}-\{1\}$. Symmetrically, for each $t \in \mathbb{N}-\{1\}$, there exist infinitely many $a : a \equiv 5 \hspace{-0.5 mm} \pmod {20}$ such that $\sqrt[c]{a} \in \mathbb{N} \wedge V(\sqrt[c]{a}) = t \wedge V(a) \geq t$ holds for all $c \in \mathbb{N}$.
\end{theorem}

\begin{proof}
Let us (constructively) prove first the last statement of Theorem~\ref{Theorem 2.3} since it simply follows from the constancy of the congruence speed as it has been shown in \cite[Section 2.1]{11}.

Let the symbol ``$\_$'' indicate the juxtaposition of consecutive digits (e.g., $3\_6\_1 = 36\_1 = 3\_61 = 361$). Consider the rightmost $t \in \mathbb{N}-\{1\}$ digits of the $10$-adic integer $\alpha_{25} \coloneq \{5^{2^n}\}_\infty$, say $x_t\_x_{t-1}\_\dots\_2\_5$, and then juxtapose to the left the $(t+1)$-th digit plus $1$ if $x_{t+1}\leq 8$ or the $(t+1)$-th digit minus $1$ if $x_{t+1}=9$.

So, let ${\tilde{x}_{t+1}} \coloneq \begin{cases} {x_{t+1}+1} \mbox{ }\mbox{ }\mbox{ }\mbox{ }\mbox{ }\mbox{ }\mbox{ }\mbox{ if \mbox{ } ${x_{t+1} \in \{0,1,2,3,4,5,6,7,8\}}$} \\
{x_{t+1}-1} \mbox{ }\mbox{ }\mbox{ }\mbox{ }\mbox{ }\mbox{ }\mbox{ }\mbox{ if \mbox{ } ${x_{t+1} \in \{9\}}$}
\end{cases}\hspace{-3.5mm}.$

Thus, the base $\tilde{a} \coloneq \tilde{x}_{t+1}\_x_{t}\_x_{t-1}\_\dots\_2\_5$ is characterized by a constant congruence speed of $t$ (i.e., $V(\tilde{x}_{t+1}\_x_{t}\_x_{t-1}\_\dots\_2\_5)=V(\tilde{a})=t$ for any $t \in \mathbb{N}-\{1\}$). This property follows from \cite[Equation~(16)]{12} (i.e., $2^t \mid \tilde{x}_{t+1}\_x_{t}\_x_{t-1}\_\dots\_2\_5 \hspace{2mm}\wedge$ $2^{t+1} \nmid \tilde{x}_{t+1}\_x_{t}\_x_{t-1}\_\dots\_2\_5$, for any $t \geq 2$). Since (as discussed in Remark~\ref{Remark 0}) $\alpha_{25} \mapsto a_{25} \coloneq 25$ and $25^2 \equiv 25 \pmod {10^2}$, from Hensel's lemma~\cite{1} (see also \cite{3, 13}), we have that if $\tilde{a} \equiv \alpha_{25} \pmod {10^t} \wedge \tilde{a} \not\equiv \alpha_{25} \pmod {10^{t+1}}$, then $\tilde{a}^c \equiv \alpha_{25} \pmod {10^t}$ for any given $c \in \mathbb{N}$ (in general, we cannot assert that $\left(\tilde{a} \equiv \alpha_{25} \pmod {10^t} \wedge \tilde{a} \not\equiv \alpha_{25} \pmod {10^{t+1}} \right)$ implies $\tilde{a}^c \not\equiv \alpha_{25} \pmod {10^{t+1}}$ for the given pair $(t, c)$).

\sloppy Consequently, by simply taking $a \coloneq (\tilde{x}_{t+1}\_x_{t}\_x_{t-1}\_\dots\_2\_5)^c$ (as $c$ is free to run over the positive integers), we have proved the existence, for any given $t \in \mathbb{N}-\{1\}$, of infinitely many tetration bases $a \equiv 5 \pmod{20}$ such that $V(a) \geq t$ holds for all the elements of the aforementioned set, a set that contains infinitely many distinct perfect powers originated from the string $\tilde{x}_{t+1}\_x_{t}\_x_{t-1}\_\dots\_2\_5$ (since $\tilde{a}$ is a positive integer by definition, then $a=\tilde{a}^c$ implies that $\sqrt[c]{a} \in \mathbb{N}$).
Hence, $V(\tilde{x}_{t+1}\_x_{t}\_x_{t-1}\_\dots\_2\_5)=t$ implies $V((\tilde{x}_{t+1}\_x_{t}\_x_{t-1}\_\dots\_2\_5)^c) \geq t$ by observing that $V((\tilde{x}_{t+1}\_x_{t}\_x_{t-1}\_\dots\_2\_5)^c)=V(\dots\_x_{t}\_x_{t-1}\_\dots\_2\_5)$, and trivially $\sqrt[c]{(\tilde{x}_{t+1}\_x_{t}\_x_{t-1}\_\dots\_2\_5)^c} \in \mathbb{N}$ for all $c \in \mathbb{N}$ (we point out that, for any $t \geq 2$ and as long as $c$ is a positive integer, $(\tilde{x}_{t+1}\_x_{t}\_x_{t-1}\_\dots\_2\_5)^c \equiv x_{t}\_x_{t-1}\_\dots\_2\_5 \pmod {10^t}$ holds by construction \cite{16}).

Now, let us prove the first statement of Theorem~\ref{Theorem 2.3} and complete the proof.

For this purpose, it is sufficient to note that
$$V \hspace{-0.7mm}\left(\tilde{x}_{t+1}\_x_{t}\_x_{t-1}\_\dots\_2\_5 \right) = V \hspace{-0.7mm}\left(10^{k+t} + \tilde{x}_{t+1}\_x_{t}\_x_{t-1}\_\dots\_2\_5 \right)$$
is true for any positive integer $k$. So, we can take the $c$-th power of every integer of the form $10^{k+t} + \tilde{x}_{t+1}\_x_{t}\_x_{t-1}\_\dots\_2\_5$ to get $k$ distinct sets of cardinality $\aleph_0$ each, whose elements, by construction, always satisfy the first statement of the theorem (we have already shown that, for any given $c \in \mathbb{N}$, if $V(10^{k+t} + \tilde{x}_{t+1}\_x_{t}\_x_{t-1}\_\dots\_2\_5)=t$, then $V((10^{k+t} + \tilde{x}_{t+1}\_x_{t}\_x_{t-1}\_\dots\_2\_5)^c) \geq t$ holds for every $t \in \mathbb{N}-\{1\}$).

Therefore, both statements of Theorem~\ref{Theorem 2.3} have been shown to be true, and this concludes the proof.
\end{proof}


\section{Main result} \label{sec:3}

From here on, let us indicate the $p$-adic valuation \cite{9} of any tetration base $a$ as $\nu_p(a)$, for any prime number $p$.

Then, we need the following lemma to prove the existence, for any $\tilde{a} \in \mathbb{N}-\{1\}$ such that $\tilde{a} \not\equiv 0 \pmod {10}$, of infinitely many $c$-th powers of $\tilde{a}$ having a constant congruence speed of $V(\tilde{a})$, $V(\tilde{a})+1$, $V(\tilde{a})+2$, $V(\tilde{a})+3$, and so forth. Furthermore, for any given positive integer $c$, Lemma~\ref{Lemma 2.5} (thanks to \cite{12}, Equation~(2), line~2) implies the existence of infinitely many tetration bases of the form $(10^{k+t}+10^t+1)^c$ (where $k\in \mathbb{N}$ and $t \in \mathbb{N}-\{1\}$) characterized by any constant congruence speed greater than or equal to $2+\min\{\nu_5(c), \nu_2(c)\}$.

\begin{lemma}\label{Lemma 2.5}
For every $c,k \in \mathbb{N}$ and $t \in \mathbb{N}-\{1\}$, $\nu_5((10^{k+t}+10^t+1)^c-1)=t+\nu_5(c)$ and $\nu_2((10^{k+t}+10^t+1)^c-1)=t+\nu_2(c)$.
\end{lemma}

\begin{proof}
First of all, we prove that, for every positive integer $c$, $\nu_5((10^{k+t}+10^t+1)^c-1)=t+\nu_5(c)$, where $t,k \in \mathbb{N}$.

This can be achieved by using the lifting-the-exponent lemma (LTE lemma).

For this purpose, let $c$ be a positive integer, we note that
$$\nu_5((10^{k+t}+10^t+1)^c-1)=\nu_5((10^{k+t}+10^t+1)^c-1^c),$$ so we can invoke the LTE lemma (see \cite{14} and \cite[Lemma 2.6]{4}) for odd primes, stating that for any integers $x$, $y$, a positive integer $c$, and a prime number $p$ such that $p \nmid x \wedge p \nmid y$, if $p$ divides $x-y$, then $\nu_p\left(x^c-y^c \right)= \nu_p(x-y)+\nu_p(c)$.

Thus, by observing that $p \coloneq 5$ is an odd prime satisfying all the conditions above for $x \coloneq 10^{k+t}+10^t+1 \wedge y \coloneq 1$ (since $10^{k+t}+10^t$ is a multiple of $5$),
$$\nu_5((10^{k+t}+10^t+1)^c-1)=\nu_5((10^{k+t}+10^t+1)^c-1^c)=\nu_5(10^{k+t}+10^t)+\nu_5(c)=t+\nu_5(c).$$

Similarly, we can use the LTE lemma to show that $2 \mid c \Rightarrow \nu_2((10^{k+t}+10^t+1)^c-1)=t+\nu_2(c)$ and also $2 \nmid c \Rightarrow \nu_2((10^{k+t}+10^t+1)^c-1)=t+\nu_2(c)$.

In detail (see \cite{15}), both $x \coloneq 10^{k+t}+10^t+1$ and $y \coloneq 1$ are odd numbers so that we can invoke the LTE lemma for $p=2$ since $2 \nmid x$, $2 \nmid y$, and $2 \mid (x-y)$.

Thus, the LTE lemma for $p=2$ states that if $c$ is odd, then $\nu_2{\left((10^{k+t}+10^t+1)^c-1^c \right)}=\nu_2{\left( (10^{k+t}+10^t+1)-1 \right)}=\nu_2{\left(10^{k+t} + 10^t \right)}=\nu_2{\left(10^t \right)} = t$.
Since $2 \nmid c \Rightarrow \nu_2(c)=0$, we can safely rewrite the above as $\nu_2{\left((10^{k+t}+10^t+1)^c-1 \right)}=t + \nu_2(c)$.

On the other hand, if $c$ is even, from the LTE lemma it follows that 
\begin{gather*}
\nu_2{\left((10^{k+t}+10^t+1)^c-1^c \right)}=\nu_2{\left( 10^{k+t}+10^t+1-1 \right)}+\nu_2{\left( 10^{k+t}+10^t+1+1 \right)} + \nu_2{(c)} - 1\\
=\nu_2{\left( 10^{k+t}+10^t \right)}+\nu_2{\left( 10^{k+t}+10^t+2 \right)} + \nu_2{(c)} - 1.
\end{gather*}

Now, for any given pair of positive integers $(r, s)$ and a prime $p$, we know that $\nu_p(r + s)= \min\{\nu_p(r), \nu_p(s)\}$ as long as $\nu_p(r) \neq \nu_p(s)$.
Accordingly, $t \geq 2$ implies $\nu_2(10^t) \geq 2 > \nu_2(2)$ so that $\nu_2{(10^{k+t}+10^t+2)}=\min\{\nu_2(10^{k+t} ), \nu_2\left(10^t \right), \nu_2(2)\} = 1 $.

Hence, $2 \mid c \Rightarrow \nu_2{\left((10^{k+t}+10^t+1)^c-1^c \right)}=t+1-1+\nu_2{(c)}=t+\nu_2{(c)}$.

Therefore, for every triad $(t,c,k)$ of positive integers, $\nu_5((10^{k+t}+10^t+1)^c-1)=t+\nu_5(c)$ and, symmetrically, $\nu_2((10^{k+t}+10^t+1)^c-1)=t + \nu_2(c)$.
\end{proof}

\begin{theorem}\label{Theorem 2.2}
For any given $c, k \in \mathbb{N}$ and $t \in \mathbb{N}-\{1\}$,
\begin{equation*}
V\hspace{-0.7mm}\left(\left(10^{k+t}+10^t+1 \right)^c \right) =t + \min\{\nu_5(c), \nu_2(c)\}.
\end{equation*}
\end{theorem}

\begin{proof}
Theorem~\ref{Theorem 2.2} easily follows from Lemma~\ref{Lemma 2.5}. Since $t \geq 2$, $10^{k+t}+10^t+1 \equiv 1 \pmod{10^2}$ (the congruence modulo $100$ does not depend on $k$ as $k+t>t$), and then also $(10^{k+t}+10^t+1)^c \equiv 1 \pmod{100}$ holds for each positive integer $c$.

By Theorem~2.1 of \cite{12} (see Equation~(2), line~2), we have that $V((10^{k+t}+10^t+1)^c)=\min\{\nu_5((10^{k+t}+10^t+1)^c-1), \nu_2((10^{k+t}+10^t+1)^c-1)\}$. Since Lemma~\ref{Lemma 2.5} asserts that $\nu_5((10^{t+k}+10^t+1)^c-1)=t+\nu_5(c)$ and $\nu_2((10^{t+k}+10^t+1)^c-1)=t+\nu_2(c)$, it follows that $V((10^{k+t}+10^t+1)^c)=\min\{ t+\nu_5(c) , t+\nu_2(c)\} = t + \min\{\nu_5(c), \nu_2(c)\}$ for any positive integers $c$, $k$, and $t-1$.
\end{proof}

\begin{remark}\label{Remark 1} \sloppy Let the tetration base $a : a \equiv 6 \pmod {10}$ be given. Then, by looking at the two rightmost digits of the corresponding $10$-adic solution, $\alpha_{76}=\dots 7109376$ (see Remark~\ref{Remark 0}), and applying the usual strategy (already described in the proof of Theorem~\ref{Theorem 2.3}), we find the sequence $a_n \coloneq 10^{n+1}+86$, $n \in \mathbb{N}$ (defining also the set $\{186,1086,10086,100086,\dots \}$). Now, we can create an infinite set consisting of the $c$-th powers of each aforementioned term, a set whose elements are all characterized by a unit constant congruence speed as long as $c$ is not a multiple of $5$.
Thus, $5 \nmid c$ implies $V(186^c)=V(1086^c)=V(10086^c)=V(100086^c)= \cdots=1$.
\end{remark}

The above is just another example of the $\tilde{x}_{t+1}$ idea, introduced in the proof of Theorem~\ref{Theorem 2.3}, shown by taking into account $t=1$ and the solution $\alpha_{76} \coloneq 1-\left\{5^{2^n}\right\}_{\infty}$ of the equation $y^2=y$ in the commutative ring of $10$-adic integers (as we know \cite{11}, the other three solutions are $\alpha_{00} \coloneq 0$, $\alpha_{01} \coloneq 1$, and $\alpha_{25} \coloneq \left\{5^{2^n}\right\}_{\infty}$).

\begin{corollary}\label{Corollary 2.1}
Let $t \in \mathbb{N}$ and assume that $V(\tilde{a})=t$. Then, for any nonnegative integer $h$, there exist infinitely many $c \in \mathbb{N}$ such that $V(\tilde{a}^c)=t+h$.
\end{corollary}

\begin{proof}
Let $k \in \mathbb{N}$. If $t \in \mathbb{N}-\{1\}$, it is sufficient to observe that, by Theorem~\ref{Theorem 2.2},
\begin{equation*}
\nu_2(c) \geq \nu_5(c) \Rightarrow V \hspace{-0.7mm}\left(\left(10^{k+t}+10^t+1 \right)^c \right) = t+\nu_5(c).
\end{equation*}

Accordingly, let $\tilde{a} \coloneq 10^{k+t}+10^t+1$, $c \coloneq 2^{k+h} \cdot 5^h$ (so that $\nu_5(2^{k+h}) \leq \nu_2(2^{k+h})$), and then, for any $h \in \mathbb{N}_0$, $\tilde{a}^c$ identifies an infinite set of valid tetration bases (since $\nu_5(2^{k+h} \cdot 5^h)=h$ is true for any positive integer $k$ and, consequently, the constant congruence speed of $\tilde{a}^{2^{k+h} \cdot 5^h}$ does not depend on $k$).

For the remaining case, $t=1$, Remark~\ref{Remark 1} gives us a valid set of solutions (i.e., any tetration base of the form $(10^{k+1}+86)^{2^{k-1} \cdot 5^h}$ does the job since $V(10^{k+1}+86)=1$ for every positive integer $k$), so just let $\tilde{a} \coloneq 10^{k+1}+86$ and $c \coloneq 2^{k-1} \cdot 5^h$ in order to get $V((10^{k+1}+86)^{2^{k-1} \cdot 5^h})=\nu_5((10^{k+1}+86)^{2^{k-1} \cdot 5^h}-1)=1+\nu_5{(5^h)}=1+h$ by \cite{12}, Equation~(16), line 4).
\end{proof}

\begin{theorem}\label{Theorem 2.4}
Let the integer $a>1$ not be a multiple of $10$. Then, for all $c \in \mathbb{N}$, there exist infinitely many $a$ such that $\sqrt[c]{a}\in \mathbb{N} \wedge V(a)=V(\sqrt[c]{a})=t$ holds for every integer $t$ greater than $\min\{\nu_5(c), \nu_2(c)\} +1$.
\end{theorem}

\begin{proof}
Since Theorem~\ref{Theorem 2.2} states that, given $t \in \mathbb{N}-\{1\}$, $V((10^{k+t}+10^t+1)^c) = t+ \min\{ \nu_5(c), \nu_2(c)\}$ for each $c, k \in \mathbb{N}$, it follows that, for any given triad $(t,k,c)$ such that $t \geq 2+\min\{\nu_5(c), \nu_2(c)\}$, at least one of the two tetration bases $(10^{k+t-\nu_5(c)}+10^{t-\nu_5(c)}+1)^c$ and $(10^{k+t-\nu_2(c)}+10^{t-\nu_2(c)}+1)^c$ is guaranteed to have a constant congruence speed of $t$.

Although this is enough to prove the theorem, we are free to simplify the generic form of the above by observing that (see \cite{15}), as we call $j \coloneq j(c)$ the last nonzero digit of $c$, only the following two cases matter: $j \neq 5$ and $j = 5$.

Hence,
\begin{align*}
j \neq 5 \Rightarrow \nu_5(c) \leq \nu_2(c) \Rightarrow V\hspace{-0.7mm}\left(\left(10^{k+t}+10^{t-\nu_5(c)}+1 \right)^c \right) &=\nu_5((10^{k+t}+10^{t-\nu_5(c)}+1 )^c-1) \\
&=\nu_5(10^{k+t}+10^{t-\nu_5(c)})+\nu_5(c) \\
&=\min\{\nu_5(10^{k+t}), \nu_5(10^{t-\nu_5(c)})\}+\nu_5(c)
\end{align*}
(since $k > -\nu_5(c)$ implies that $\nu_5(10^{k+t}) > \nu_5(10^{t-\nu_5(c)})$), and then it follows that $$V\hspace{-0.7mm}\left(\left(10^{k+t}+10^{t-\nu_5(c)}+1 \right)^c \right)=\nu_5\left(10^{t-\nu_5(c)}\right)+\nu_5(c)=t-\nu_5(c)+\nu_5(c)=t.$$

Conversely (given the fact that $\nu_2(10^{k+t}) > \nu_2(10^{k+t-\nu_2(c)}) > \nu_2(10^{t-\nu_2(c)})$),
\begin{align*}
j = 5 \Rightarrow \nu_5(c) > \nu_2(c) \Rightarrow V\hspace{-0.7mm}\left(\left(10^{k+t}+10^{t-\nu_2(c)}+1 \right)^c \right) &=\nu_2((10^{k+t}+10^{t-\nu_2(c)}+1)^c-1) \\
&=\nu_2((10^{k+t-\nu_2(c)}+10^{t-\nu_2(c)}+1)^c-1)
\end{align*}
and, by Lemma~\ref{Lemma 2.5}, $\nu_2((10^{k+t-\nu_2(c)}+10^{t-\nu_2(c)}+1)^c-1)=t$ follows.

Thus, $c : j\neq 5$ guarantees that $V((10^{k+t}+10^{t-\nu_5(c)}+1)^c) = t$ is true for any $c,k \in \mathbb{N}$ and $t>\nu_5(c)$, while $V((10^{k+t}+10^{t-\nu_2(c)}+1)^c) = t$ as $c : j = 5$ and $t>\nu_2(c)$.

Therefore, we have shown that, for any given positive integer $c$, at least one of \linebreak $V((10^{k+t}+10^{t-\nu_5(c)}+1)^c)=t$ ($t>\nu_5(c)+1$) and $V((10^{k+t}+10^{t-\nu_2(c)}+1)^c)=t$ ($t>\nu_2(c)+1$) is true, so the proof is complete.
\end{proof}

Now, let $t = c$ and observe that $c - \nu_5(c) < 2 \Rightarrow c = 1$, while $c - \nu_2(c) < 2 \Rightarrow c \leq 2$.

Then, the proof of Theorem~\ref{Theorem 2.4} shows the existence of a very special set of tetration bases that are $c$-th powers of an integer and whose constant congruence speed is $c \in \mathbb{N}-\{1,2\}$, a set including all the bases of the form $(10^{c+k}+10^{c-\min\{\nu_5(c), \hspace{0.5 mm} \nu_2(c)\}}+1)^c$, $k \in \mathbb{N}$ (e.g., $(c=1000 \wedge k=314)$ implies $V((10^{1314}+10^{997}+1)^{1000})=1000$).

Therefore, for each $c > 2$, we have already proved Corollary~\ref{Corollary 2.5} as a particular case of Theorem~\ref{Theorem 2.4}.

\begin{corollary}\label{Corollary 2.5} Let $a \in \mathbb{N}-\{1\}$ not be a multiple of $10$. Then, for every $c \in \mathbb{N}$, there exist infinitely many $a$ such that $\sqrt[c]{a}\in \mathbb{N} \wedge V(a)=V(\sqrt[c]{a})=c$.
\end{corollary}

Indeed, setting $S$ as the sum of digits function to base $10$ and assuming $c > 2$,\linebreak we have $S(10^{c+k}+10^{c-\min\{\nu_5(c), \hspace{0.5 mm} \nu_2(c)\}}+1)=3$, $3 \mid \left(10^{c+k}+10^{c-\min\{\nu_5(c), \hspace{0.5 mm} \nu_2(c)\}}+1 \right)$ and $3^2 \nmid \left(10^{c+k}+10^{c-\min\{\nu_5(c), \hspace{0.5 mm} \nu_2(c)\}}+1 \right)$ (implying that none of these numbers can be a perfect power of degree greater than $1$), so there exist infinitely many perfect powers of degree $c=3,4,5,6 \ldots$ having a constant congruence speed of $c$, and it is sufficient to consider the tetration bases of the form $(10^{c+k}+10^{c-\min\{\nu_5(c), \hspace{0.5 mm} \nu_2(c)\}}+1)^c$ (see the OEIS sequence \seqnum{A379243} for $k \coloneq 1$).

On the other hand, if $c \in \{1,2\}$, then the second last nonzero digit of any tetration base of the form $(10^{c+k}+10^{c-\nu_5(c)}+1)^c$ will not be equal to $5$ (we note that the second last nonzero digit of each $c$-th perfect power of $10^{k+t}+10^t+1$ is always equal to the last nonzero digit of $c$), and thus we can cover the remaining cases $c=1$ and $c=2$ by taking $10^{1+k}+10^1+1$ ($V((10^{1+k}+11)^1 )=1$ holds for each positive integer $k$ by \cite{12}, Equation~(16), line~9) and $(10^{2+k}+10^2+1)^2$ (respectively).

Since $\sqrt[c]{10^{1+k}+11} \in \mathbb{N} \Rightarrow c=1$ and $\sqrt[c]{(10^{2+k}+101)^2}\in \mathbb{N} \Rightarrow c \in \{1,2\}$, we have finally proved for any given positive integer $c$ the existence of infinitely many tetration bases $\tilde{a}^c$ which are perfect powers of degree $c$ (exactly) and such that $V(\tilde{a}^c)=V(\tilde{a})=c$.

\begin{remark}\label{Remark 2}
Let $c,d,k,t \in \mathbb{N}$ be such that $10^{d-1} \leq c < 10^d$ and $t \geq d+1$. Then, $(10^{k+t}+10^t+1)^c \equiv c \cdot 10^t+1 \pmod{10^{t+d}}$ easily follows from the trinomial expansion of $(10^{k+t}+10^t+1)^c$. Furthermore, as we call $j$ the last nonzero digit of $c$ (e.g., $j(9400{\color{red} 3}0)={\color{red} 3}$), we observe that it is sufficient to set $t \geq \min\{\nu_5(c), \nu_2(c)\}+ 2$ in order to achieve $V((10^{k+t}+10^{t-\nu_5(c)}+1)^c)=t$ if $c : j \nmid 5$ (by \cite{12}, Equation~(16), line~7) and $V((10^{k+t}+10^{t-\nu_2(c)}+1)^c)=t$ otherwise (by~\cite{12}, Equation~(16), line~8).
In particular, if $2 \nmid c$ or $5 \nmid c$, then $V((10^{k+c}+10^c+1)^c)=c$ (since $\min\{\nu_5(c), \nu_2(c)\}=0 \iff c \not\equiv 0 \pmod{10}$).
\end{remark}


\section{Conclusion} \label{sec:Concl}
\sloppy In the previous section, for each $c \in \mathbb{N}$, we have provided the general equation
\begin{equation*}
V \hspace{-0.7mm}\left(\left(10^{k+t}+10^{t-\min\{\nu_5(c), \hspace{0.5 mm} \nu_2(c)\}}+1 \right)^c \right)=t \quad (k,t \in \mathbb{N} : t>\min\{\nu_5(c),\nu_2(c)\}+1),
\end{equation*}
which shows the existence of infinitely many perfect powers of degree $c$ with a constant congruence speed of $\min\{\nu_5(c), \nu_2(c)\}+2$, $\min\{\nu_5(c), \nu_2(c)\}+3$, $\min\{\nu_5(c), \nu_2(c)\}+4$, $\min\{\nu_5(c), \nu_2(c)\}+5$, and so forth.

In conclusion, for any given positive integer $c$, we have constructed an infinite set of $c$-th perfect powers that are also characterized by a constant congruence speed of $c$.


\newpage
\noindent {\bf Acknowledgement.} The author would like to thank his friend Aldo Roberto Pessolano for having shared with him his first thoughts about the outcome of the computer search performed on the smallest integers of the form $a^c$, congruent to $5$ modulo $20$, and characterized by a constant congruence speed of $c$.

\bigskip

\noindent (Concerned with sequences
\seqnum{A018247},
\seqnum{A018248},
\seqnum{A317905}, 
\seqnum{A372490}, and
\seqnum{A379243}.)

\end{document}